\documentclass[10pt,a4paper,oneside,english]{amsart}
\usepackage[T1]{fontenc}
\usepackage[latin9]{inputenc}
\usepackage{units}
\usepackage{amsthm}
\usepackage{amssymb}

\makeatletter

\numberwithin{equation}{section}
\numberwithin{figure}{section}
\theoremstyle{plain}
\newtheorem{thm}{\protect\theoremname}[section]
  \theoremstyle{plain}
  
  \theoremstyle{plain}
  \newtheorem{prop}[thm]{\protect\propositionname}
  \theoremstyle{remark}
  \newtheorem{rem}[thm]{\protect\remarkname}
  \theoremstyle{plain}
  \newtheorem{cor}[thm]{\protect\corollaryname}

\@ifundefined{date}{}{\date{}}

\usepackage{amsthm}

\usepackage{cite}

\newtheorem{theorem}{Theorem}[section]

\newtheorem{definition}[theorem]{Definition}
\newtheorem{example}[theorem]{Example}

\providecommand{\lemmaname}{Lemma}
  \providecommand{\propositionname}{Proposition}
  \providecommand{\remarkname}{Remark}
\providecommand{\theoremname}{Theorem}

\usepackage{babel}
\providecommand{\lemmaname}{Lemma}
  \providecommand{\propositionname}{Proposition}
  \providecommand{\remarkname}{Remark}
\providecommand{\theoremname}{Theorem}

\usepackage{babel}
\providecommand{\corollaryname}{Corollary}
  \providecommand{\lemmaname}{Lemma}
  \providecommand{\propositionname}{Proposition}
  \providecommand{\remarkname}{Remark}
\providecommand{\theoremname}{Theorem}

\makeatother

\usepackage{babel}
  \providecommand{\corollaryname}{Corollary}
  \providecommand{\lemmaname}{Lemma}
  \providecommand{\propositionname}{Proposition}
  \providecommand{\remarkname}{Remark}
\providecommand{\theoremname}{Theorem}

\begin{document}

\title[Lindenstrauss spaces whose duals lack the $w^*$-FPP]{Separable Lindenstrauss spaces whose duals lack the weak$^*$
fixed point property for nonexpansive mappings}

\author{E. Casini}

\address{Dipartimento di Scienza e Alta Tecnologia, Università dell'Insubria,
via Valleggio 11, 22100 Como, Italy }

\email{emanuele.casini@uninsunbria.it}

\author{E. Miglierina}

\address{Dipartimento di Discipline Matematiche, Finanza Matematica ed Econometria,
Università Cattolica del Sacro Cuore, Via Necchi 9, 20123 Milano,
Italy }

\email{enrico.miglierina@unicatt.it}

\author{\L. Piasecki}

\address{Instytut Matematyki,
Uniwersytet Marii Curie-Sk{\l}odowskiej, Pl. Marii Curie-Sk{\l}odowskiej 1, 20-031 Lublin, Poland}

\email{piasecki@hektor.umcs.lublin.pl}
\begin{abstract}
In this paper we study the $w^*$-fixed point property for
nonexpansive mappings. First we show that the dual space $X^*$ lacks
the $w^*$-fixed point property whenever $X$ contains an isometric
copy of the space $c$. Then, the main result of our paper provides
several characterizations of weak-star topologies that fail the
fixed point property for nonexpansive mappings in $\ell_1$ space.
This result allows us to obtain a  characterization of all separable
Lindenstrauss spaces $X$ inducing the failure of $w^*$-fixed point
property in $X^*$.
\end{abstract}

\subjclass[2010]{Primary 47H09; Secondary 46B25}

\keywords{nonexpansive mappings, $w^*$-fixed point property, Lindenstrauss spaces }

\maketitle

\section{Introduction}
Let $X$ be an infinite dimensional real Banach space and let us denote by $B_X$ its closed unit
ball. A nonempty bounded closed and convex subset $C$ of $X$ has the fixed point property
(shortly, FPP) if each nonexpansive mapping (i.e., the mapping $T:C\rightarrow C$ such that $\|T(x)-T(y)\|\leq \|x-y\|$ for all $x,y\in C$) has a fixed point.
The space $X^*$ is said to have the $\sigma(X^*,X)$-fixed
point property ($\sigma(X^*,X)$-FPP) if every nonempty, convex,
$w^*$-compact subset $C$ of $X^*$ has the FPP. The study of the
$\sigma(X^*,X)$-FPP reveals to be of special interest whenever a
dual space has different preduals. Indeed, the behaviour with
respect to the $\sigma(X^*,X)$-FPP of a given dual space can be
completely different if we consider two different preduals. For
instance, this situation occurs when we consider the space $\ell_1$
and its preduals $c_0$ and $c$ where it is well-known (see
\cite{Karlovitz1976}) that $\ell_1$ has the $\sigma(\ell_1,c_0)$-FPP
whereas it lacks the $\sigma(\ell_1,c)$-FPP.

The main aim of this paper is
to study some structural features of a separable space $X$ linked to $\sigma(X^*,X)$-FPP on its dual.

At the beginning of Section 3, we state a sufficient condition for
the failure of $\sigma(X^*,X)$-FPP. Indeed, Theorem \ref{thm c
inside predual} shows that the presence of an isometric copy of $c$
in a separable space $X$ implies the failure of the
$\sigma(X^*,X)$-FPP. This theorem extends a result of Smyth (Theorem
1 in \cite{Smyth1994}) to a broader class of spaces. Moreover, it
allows us to show that all the separable Lindenstrauss spaces $X$
(i.e., the space such that its dual is a space $L_1(\mu)$  for some
measure $\mu$ ), whose duals are nonseparable, lack the
$\sigma(X^*,X)$-FPP. Taking into account these last facts, it seems to
be natural to investigate if the presence of an isometric copy of
$c$ in the space $X$ is also a necessary condition for the failure
of $\sigma(X^*,X)$-FPP. The simple example where $X=\ell_1$ shows
that the answer is negative in a general framework. Moreover, by
considering a suitable class of hyperplanes of $c$, we are able to
show that the answer remains negative even if we add the assumption
that $X$ is a separable Lindenstrauss space. This class of
hyperplanes of $c$ with duals isometric to $\ell_1$ and failing the
$w^*$-FPP, will play an important role in this paper and
subsequently their elements will be indicate as {\it "bad" $W_f$}
(see Section \ref{section W_f} for a detailed description of these
spaces). The first interesting result involving this class of spaces
is Theorem \ref{W_f in a predual of X}, where we prove that if a
separable space $X$ contains a {\it "bad" $W_f$} then
$\sigma(X^*,X)$-FPP still fails. A simple but relevant consequence
of this theorem is Remark \ref{W_f + quotient} where it is stated
that the dual space $X^*$ lacks the $\sigma(X^*,X)$-FPP whenever
there is a quotient of $X$ that contains an isometric copy of a
"bad" $W_f$. The last section is devoted to the characterization of
the $w^*$-topologies that lack the $\sigma(\ell_1,X)$-FPP. Theorem
\ref{OGT}, which is the main theorem of this paper, lists several
properties of a predual $X$ of $\ell_1$ that are all equivalent to
the lack of $\sigma(\ell_1,X)$-FPP for $\ell_1$. Among them, one
property is exactly the structural condition appeared in Remark
\ref{W_f + quotient}. Another property listed in this result (see
condition (\ref{item subsequence}) in Theorem \ref{OGT}) seems to be
meaningful. It is related to the $w^*$-cluster points of the
standard basis of $\ell_1$ and it allows us to extend Theorem $8$ in
\cite{Japon-Prus2004} in the case of $w^*$-topologies. Indeed, we
can prove this theorem without assuming the strong assumption on
$w^*$-convergence of the standard basis of $\ell_1$ that was made in
\cite{Japon-Prus2004}.

Throughout all the paper we will follow the standard terminology and
notations. In particular, it is well-known that $c^*$ can be
isometrically identified with $\ell_1$ in the following way. For
every $x^*\in c^*$ there exists a unique element $f=(f(1),f(2),\dots)\in \ell_1$
such that
\[
x^*(x)=\sum_{n=0}^\infty f(n+1)x(n)=f(x)
\]
with $x=(x(1),x(2),\dots)\in c$ and $x(0)=\lim x(n)$.

\section{A class of hyperplanes in the space of convergent sequences}\label{section W_f}
This section is devoted to recall some properties of a class of hyperplanes of $c$ that will play a crucial role in the sequel of our paper. For the convenience of the reader we repeat some materials from \cite{Casini-Miglierina-Piasecki2014} without proofs, thus making our exposition self-contained. Moreover, we prove some additional properties of these hyperplanes directly related to the topic studied in the present paper.

Let $f\in\ell_1=c^*$ be such that $\|f\|=1$. We
consider the hyperplane of $c$ defined by
\[
W_f=\left\lbrace x\in c:f(x)=0 \right\rbrace.
\]
In \cite{Casini-Miglierina-Piasecki2014}, the following results are proved:
\begin{enumerate}
\item[(I)] there exists $j_0\geq1$ such that $|f(j_0)|\geq 1/2$ if and only if $W_f^*$ is isometric to $\ell_1$.
\item[(II)] there exists $j_0\geq2$ such that $|f(j_0)|\geq 1/2$ if and only if $W_f$ is isometric to $c$.

\end{enumerate}

For our aims, an important case to be considered is when $|f(1)|\geq 1/2$ and $|f(j)|< 1/2$ for every $j\geq2$.
Under these additional assumptions, Theorem 4.3 in \cite{Casini-Miglierina-Piasecki2014} identifies $W_f^*$ and $\ell_1$ by giving the following dual action: for every $x^*\in W_f^*$ there exists a unique element $g\in\ell_1$ such that
\begin{equation}\label{dualityW}
x^*(x)=\sum_{n=1}^\infty g(n)x(n)=g(x)
\end{equation}
where $x=(x(1),x(2),\dots)\in W_f$. We conclude this section by proving some
additional properties of the spaces $W_f$ that will be useful in the
sequel. The first proposition gives a necessary and sufficient
condition for the existence of a subspace of $W_f$ isometric to $c$.
\begin{prop}\label{c inside predual}
Let $f\in \ell_1=c^*$ be such that
$\left\|f\right\|=1$ and $\left|f(1)\right|\geq \frac12$. Then the
following statements are equivalent.
\begin{enumerate}
\item $W_f$ contains a subspace isometric to $c$.
\item $\left|f(1)\right|= \frac12$, $\left\{n\in \mathbb{N}:f(1) f(n+1)>0\right\}$
is a finite set and $\left\{n\in \mathbb{N}:f(n+1)=0\right\}$ is an
infinite set.
\end{enumerate}
\end{prop}

\begin{proof} (2)$\Longrightarrow$(1). Let $\left\{n\in \mathbb{N}:f(n+1)=0\right\}=\left\lbrace n_k\right\rbrace _{k=1}^{+\infty}$  and let us consider the mapping $T:c\rightarrow W_f$ defined for every
$x=(x(1),x(2),\dots)\in c$ by $T(x)=((T(x))(1),(T(x))(2),\dots)\in W_f$, where
 \begin{equation*}
    (T(x))(i)=
        \begin{cases}
            x(k) & \quad\text{if}\quad i=n_k,\\
            -\textrm{sgn}(f(1)f(i+1))\cdot \lim\limits_{j}x(j) & \quad\text{if}\quad i\in \mathbb{N}
            \setminus \left\{n_k\right\}.
        \end{cases}
  \end{equation*}
It is easy to see that $T$ is a linear isometry of $c$ into $W_f$.

(1)$\Longrightarrow$(2). If $W_f$ is isometric to $c$, then the
thesis follows immediately from the result recalled in item (II) at
the beginning of this section. Suppose that $W_f$ is not isometric
to $c$. Let $\left(e^*_n \right)_{n\geq 1}$ be the standard basis of
$\ell_1=c^*$. For every $n\geq 2$ we take a norm-one extension of
$e^*_n$ to the whole space $W_f$ and we denote it by $g^*_n$.
Consider a $\sigma(\ell_1,W_f)$-convergent subsequence
$(g_{n_{k}}^*)_{k\geq 2}$ of $(g_n^*)_{n\geq 2}$ and denote its
limit by $g_{n_1}^*$. Obviously, $g_{n_1}^*$ is a norm-one extension
of $e^*_1$ to the whole $W_f$. It is easy to see that
$\left\|g^*_{n_k}\pm g^*_{n_l}\right\|=2$ for all $k,l \in
\mathbb{N}$, $k\neq l$. Consequently,
\begin{equation}\label{disjoint support}
\mbox{supp }g^*_{n_k} \cap \mbox{supp }g^*_{n_l}=\emptyset
\end{equation}
for all $k,l \in \mathbb{N}$, $k\neq l$, where $\mbox{supp
}g^*_n:=\left\{i\in \mathbb{N}:g^*_n(i)\neq 0\right\}$. Hence, by
using the argument presented at the beginning of the proof of
Theorem 8 in \cite{Japon-Prus2004} and Theorem 4.3 in
\cite{Casini-Miglierina-Piasecki2014}, we obtain
$$g^*_{n_1}=\pm \left(\frac{f(2)}{f(1)},\frac{f(3)}{f(1)},\frac{f(4)}{f(1)},\dots\right).$$
Therefore we have that $\left|f(1)\right|= \frac12$ and $\left\{n\in
\mathbb{N}:f(n+1)=0\right\}$ is an infinite set. Since there exists
$x\in c \subset W_f$ such that $\|x\|=1$ and $e^*_n(x)=1$ for every
$n$, we get
$$e^*_{n_k}(x)=g^*_{n_k}(x)=g^*_{n_1}(x)=1$$ for every $k \geq 2$.
From the above relation and the standard duality of $W_f$ (see
(\ref{dualityW}) above) we have
\begin{equation}\label{evaluation g}
x(i)=\mbox{sgn}(g^*_{n_k}(i))
\end{equation}
for every $i\in \mbox{supp }g^*_{n_k}$ and for every $k \in
\mathbb{N}$. Taking into account (\ref{disjoint support}) and
(\ref{evaluation g}) we conclude that there exists $ i_0 $ such that
either $ x(i)=1 $ or $ x(i)=-1 $ for infinitely many $ i\geq i_0 $.
Therefore $ \{n\in \mathbb{N}:f(1)f(n+1)>0\} $ is a finite set.
\end{proof}

The last proposition of this section characterizes a class of spaces $W_f$ that enjoy the $\sigma(\ell_{1},W_{f})$-FPP.

\begin{prop} \label{fixpoint theorem}Let
$f\in \ell_1=c^*$ be such that
$\left\|f\right\|=1$, $\frac{1}{2}\leq\left|f(1)\right|<1$ and
$\left|f(j)\right|<\frac{1}{2}$ for every $j\geq2$. The space $\ell_{1}$ has the
$\sigma(\ell_{1},W_{f})$-FPP if and only if one of
the following conditions holds
\begin{enumerate}
\item $\left|f(1)\right|>\frac{1}{2}$
\item $\left|f(1)\right|=\frac{1}{2}$ and the set $N^+=\left\{ n\in\mathbb{N}:f(1) f(n+1)\leq0\right\} $
is finite.
\end{enumerate}
\end{prop}
\begin{proof}
As recalled at the beginning of this section (see item (I) above),
we have that $W_f^*=\ell_1$. Now, Theorem 4.3 in
\cite{Casini-Miglierina-Piasecki2014} shows that $$
e_{n}^*\overset{\sigma(\ell_1,W_f)}{\longrightarrow}e^*, $$ where
$e^*=\left(-f(2)/f(1),-f(3)/f(1),\dots\right)$. The conclusion follows
immediately from Theorem 8 in \cite{Japon-Prus2004}.
\end{proof}

Since the spaces $W_f$ lacking the $w^*$-FPP play a crucial role in our study
 we introduce the following definition, suggested by Proposition \ref{fixpoint theorem} and item (II).
\begin{definition}\label{definition bad W_f}
A space $W_f$ is called "bad with respect to $w^*$-FPP" (shortly
"bad") if $f\in\ell_1$ is such that $\left\| f\right\|=1$,
$\left| f(1)\right|=\frac{1}{2}$ and the set $N^+=\left\lbrace n\in
\mathbb{N}:f(1) f(n+1) \leq 0\right\rbrace $ is infinite.
\end{definition}

We underline that, by combining Propositions \ref{c inside predual}
and \ref{fixpoint theorem}, we find an example of a $\ell_1$-predual
space $X$ such that $\ell_1$ fails the $\sigma(\ell_1,X)$-FPP but
$X$ does not contain an isometric copy of $c$.

\begin{example}\label{example bad W-f}
Let us consider the space $W_f$ where
\[
f=\left( \dfrac{1}{2}, -\dfrac{1}{4}, \dfrac{1}{8},
-\dfrac{1}{16},\dots\right)\in\ell_1.
\]We have that
\begin{itemize}
\item $W_f^*=\ell_1$;
\item $W_f$ does not contain an isometric copy of $c$ (by Proposition \ref{c inside predual});
\item $\ell_1$ lacks the $\sigma(\ell_1,W_f)$-FPP (by Proposition \ref{fixpoint theorem}).
\end{itemize}
We point out another feature of this space that will be useful in
the last section. The space $W_f$ does not have a quotient that
contains an isometric copy of $c$. Indeed, let us suppose
$c\subseteq W_f/Y$. Then, by following the reasoning from the proof
of Proposition \ref{c inside predual}, we obtain a sequence $\left(
x_n^*\right)_{n\geq1}\subset {(W_f/Y)}^*$ such that

\begin{itemize}
\item $x_{n}^{*}\overset{\sigma((W_{f}/Y)^*,W_{f}/Y)}{\longrightarrow}x_{1}^{*}$,
\item $\left\|x_n^*\right\|=1$ for every $n\in \mathbb{N}$,
\item $\left\|x_n^*\pm x_m^*\right\|=2$ for all $m,n \in
\mathbb{N}$, $m\neq n$.
\end{itemize}
Now, for each $u\in v+Y$, $v \in W_f$, we put $y^*_n(u)=x^*_n(v+Y)$.
Consequently, the sequence $\left( y_n^*\right)_{n\geq1}\subset
W_f^*$ is equivalent to the standard basis in $\ell_1$ and
$y_{n}^{*}\overset{\sigma(\ell_{1},W_{f})}{\longrightarrow}y_{1}^{*}$.
Again, by following the argument developed at the beginning of the
proof of Theorem 8 in \cite{Japon-Prus2004}, Theorem 4.3 in
\cite{Casini-Miglierina-Piasecki2014} yields
\[
y^*_1=\pm\left( \dfrac{1}{2}, -\dfrac{1}{4}, \dfrac{1}{8},
-\dfrac{1}{16},\dots\right).
\]
The last equality gives a contradiction.

\end{example}

We conclude this section by relating Proposition \ref{fixpoint theorem} to some results existing in the literature.

\begin{rem}
If we restrict our attention to $w^*$-topologies on $\ell_1$, the assumptions of Theorem 8 in \cite{Japon-Prus2004} are equivalent to those of Proposition \ref{fixpoint theorem}. Indeed, if $X$ is a predual of $\ell_1$ such that
the standard basis of $\ell_1$ is a $\sigma(\ell_1,X)$-convergent sequence, then there exists a suitable $W_f$ isometric to $X$ (see Corollary 4.4 in \cite{Casini-Miglierina-Piasecki2014}).
\end{rem}

\begin{rem}
In the case of a particular family of sets in $\ell_1$, a
characterization of the fixed point property for nonexpansive
mappings was established in \cite{Goebel-Kuczumow1978}. For example,
for every $\varepsilon \in (0,1)$ we define the set
$C_{\varepsilon}\subset \ell_1$ by
$$C_{\varepsilon}=\left\{\alpha_1
(1-\varepsilon)e_1^*+\sum_{i=2}^{\infty}\alpha_i e_i^*: \alpha_i
\geq 0, \sum_{i=1}^{\infty}\alpha_i=1\right\}.$$ The set
$C_{\varepsilon}$ is convex, bounded and closed. Moreover, it has
the FPP (see \cite{Goebel-Kuczumow1978}). Obviously
$C_{\varepsilon}$ is neither $\sigma(\ell_1,c)$-compact nor
$\sigma(\ell_1,c_0)$-compact.

Let $f=\left(\frac{1}{2-\varepsilon},-
\frac{1-\varepsilon}{2-\varepsilon},0,0,\dots \right)$, from Theorem
4.3 in \cite{Casini-Miglierina-Piasecki2014} we know that
$W_f^*=\ell_1$ and
$$e_{n}^{*}\overset{\sigma(\ell_{1},W_{f})}{\longrightarrow}(1-\varepsilon)e_{1}^{*}.$$
Hence, Corollary 2 in \cite{Japon-Prus2004} implies that
$C_{\varepsilon}$ is $\sigma(\ell_1,W_f)$-compact. By Proposition
\ref{fixpoint theorem}, $\ell_1$ has the $\sigma(\ell_1,W_f)$-FPP.
\end{rem}

\section{Sufficient conditions for the lack of weak$^*$ fixed point property in the dual of a separable Banach space}\label{Section separable}
This section is devoted to find some sufficient conditions for the
lack of $\sigma(X^*,X)$-FPP where $X$ is a separable space. The
first step is suggested by the well-known example of $X=c$. Indeed,
we start by showing that the presence in $X$ of a copy of $c$
implies the failure of $\sigma(X^*,X)$-FPP. In order to prove this
theorem we use an auxiliary result about the existence of a
$1$-complemented copy of $c$.

\begin{prop} \label{c in separable X}
Let $X$ be a separable Banach space that contains an isometric copy
of $c$. Then there is a subspace $Y$ of $X$ such that $Y$ is
isometric to $c$ and $1$-complemented in $X$.
\end{prop}

\begin{proof}
Let $\left(e_n^*\right)_{n\geq1}$ be the standard basis of
$c^*=\ell_1$. For each $n\in \mathbb{N}$, we consider a norm
preserving extension of $e_n^*$ to the whole $X$, we denote it by
$x_n^*$. Then, there exists a subsequence $\left(x_{n_j}^*\right)$
such that $n_1 >1$ and

\[
x_{n_j}^*\overset{\sigma(X^*,X)}{\longrightarrow}\overline{x}^*.
\]
Let us consider the subspace
\[
Y=\left\{y\in c: \lim y(n)=y(0)=y(s) \textrm{ for each } s \in \mathbb{N}
\setminus \left\{n_j-1\right\} \right\}
\]
and the mapping $P:X\rightarrow Y$ defined by

\[
P(x)=\overline{x}^*(x)e_0+\sum_{j=1}^{\infty}\left(x_{n_{j}}^*-\overline{x}^*\right)(x)e_{n_j-1},
\]
where $e_0=(1,1,\dots,1,\dots)$. It is easy to see that $Y$ is
isometric to $c$ and $P$ is a norm-one projection onto $Y$.
\end{proof}

\begin{thm} \label{thm c inside predual}
Let $X$ be a separable Banach space that contains a subspace
isometric to $c$. Then $X^*$ fails $\sigma(X^*,X)$-FPP.
\end{thm}

\begin{proof}
By Proposition \ref{c in separable X} we may assume that $c$ is
$1$-complemented in $X$. So, there is a projection $P$ of $X$ onto
$c$ with $\left\|P\right\|=1$. Then, $P^*:c^*\rightarrow X^*$ is a
$w^*$-continuous isometry. Since $c^*$ fails to have the $w^*$-FPP, there exists a $w^*$-compact convex set $C$ that lacks the FPP. Therefore $P^*(C)$ is a convex,
$w^*$-compact set in $X^*$ which lacks the FPP.
\end{proof}

\begin{rem}\label{lack of w^*-FPP for contractive maps}
It is easy to find  a $w^*$-compact and convex set $C\subset c^*$
which fails the FPP for isometry. Moreover, Lennard (see Ex.
3.2-3.3, pp. 41-43 in \cite{Handbook}) found an example of a convex,
$w^*$-compact set $C\subset c^*$ that fails the FPP for affine (as
well as for non affine) contractive mappings (i.e., the mappings
$T:C\rightarrow C$ such that $\|T(x)-T(y)\|<\|x-y\|$ for all $x,y\in
C, x\neq y$). Therefore, under the same assumptions of the previous
theorem,  $X^*$ fails $\sigma (X^*,X)$-FPP for isometries and affine
contractive mappings.
\end{rem}

A consequence of Theorem \ref{thm c inside predual} shows that all the separable Lindenstrauss spaces with a nonseparable dual fail to have the $w^*$-FPP.

\begin{cor}\label{cor-nonseparable Lspace}
Let $X$ be a separable Lindenstrauss space such that $X^*$ is a non-separable space. Then $X^*$ lacks the $\sigma(X^*,X)$-FPP.
\end{cor}
\begin{proof}
Theorem 2.3 in \cite{Lazar-Lindenstrauss1971} proves that a separable Lindenstrauss space $X$ with nonseparable dual contains a subspace isometric to the space $\mathcal{C}(\Delta)$ where $\Delta$ is the Cantor set. Since $\mathcal{C}(\Delta)$ contains an isometric copy of $c$, the thesis follows directly from Theorem \ref{thm c inside predual}.
\end{proof}

A simple extension of Theorem \ref{thm c inside predual} can be easily obtained by considering a quotient of $X$ instead of a subspace.

\begin{rem}\label{remark quotient c}
Let $X$ be a separable Banach space and let us suppose that there exists a quotient $X/Y$ of $X$ isometric to $c$. Theorem \ref{thm c inside predual} shows that $Y^\bot$ fails the $w^*$-FPP and it follows easily that also $X^*$ fails the $w^*$-FPP.
\end{rem}
 The following example shows that to consider a quotient of $X$ is a true extension of Theorem \ref{thm c inside predual}.
 \begin{example}\label{example not c not quotient}
 Let us consider the space $W_f$ where
 $$ f=\left(-\frac{1}{2},\frac{1}{4},0,-\frac{1}{8},0,\frac{1}{16},0,\dots\right)\in \ell_1. $$
 We have that
 \begin{itemize}
 \item $W_f^*=\ell_1$;
 \item $W_f$ does not contain an isometric copy of $c$ (by Proposition \ref{c inside predual});
 \item $\ell_1$ lacks the $\sigma(\ell_1,W_f)$-FPP (by Proposition \ref{fixpoint theorem}).
 \end{itemize}
 Moreover, there exists a quotient of $W_f$ isometric to $c$. Indeed, let us consider the subspace 
 $$Y=\left\lbrace y \in W_f: y(2k)=0 \mbox{ for all } k\in \mathbb{N} \right\rbrace$$ and the map $T:c\longrightarrow W_f/Y$ defined by 
 $$ T(x)=\left( \frac{7}{3}x(0),x(1),x(0),x(2),x(0),\dots\right)+Y$$
 for every $x\in c$. The map $T$ is easily seen to be a surjective isometry.
 \end{example}

It is easy to observe that the lack of $w^*$-FPP does not imply that $c\subset X$ when $X$ is a generic separable Banach space. Indeed, the well-known example by Alspach \cite{A1981}
shows that $\ell_{\infty}$ fails the $\sigma (\ell_{\infty},\ell_1)$-FPP, whereas its only predual
does not contain an isometric copy of $c$.
Moreover, Example \ref{example bad W-f} shows that also a Lindenstrauss space exhibits the same behaviour. The same example proves that also the lack of a quotient of $X$ containing an isometric copy of $c$ is not a necessary condition for the lack of $w^*$-FPP.   

The next result extends Theorem \ref{thm c inside predual}.
Indeed, the space $c$ can be regarded as a special member of the
family of "bad" $W_f$ by taking $f=\left(
\frac{1}{2},\frac{1}{2},0,0,\dots\right)$ (see Section \ref{section
W_f}).

\begin{thm}\label{W_f in a predual of X}
Let $X$ be a separable Banach space. If $X$ contains a subspace isometric to a "bad" $W_f$, then $X^*$ fails the $\sigma(X^*,X)$-FPP.
\end{thm}

\begin{proof}
Let $x\in W_{f}$ and $\left( e_{n}^{*}\right) $
be a sequence of elements of $W_{f}^{*}$ defined by
\[
e_{n}^{*}(x)=x(n)
\]
for every $n\in\mathbb{N}$. From Theorem 4.3 in
\cite{Casini-Miglierina-Piasecki2014} we have that
\[
e_{n}^{*}\overset{\sigma(\ell_{1},W_{f})}{\longrightarrow}e^{*}
\]
where
$e^{*}=\left(-\frac{f(2)}{f(1)},-\frac{f(3)}{f(1)},-\frac{f(4)}{f(1)},\dots\right)$
(observe that the same relation holds when
$\left|f(j)\right|=\frac12$ for some $j\geq 2$). We denote by
$x_{n}^{*}$ the equal norm extensions to the whole space $X$ of the
functionals $e_{n}^{*}$. By the assumption about $W_{f}$ we have
that the set $N^{+}=\left\{
n\in\mbox{\ensuremath{\mathbb{N}}}:f(1)f(n+1)\leq0\right\} $ has
infinitely many elements. Therefore we can choose an increasing
sequence $\left( n_{j}\right) \subset N^{+}$ such that
\[
x_{n_{j}}^{*}\overset{\sigma(X^{*},X)}{\longrightarrow}x^{*}
\]
and $w_{0}=e^{*}-u_{0}\neq0$ where
$u_{0}=\sum_{j=1}^{+\infty}e^{*}(n_{j})e_{n_{j}}^{*}$. Now we
consider the extension of $u_{0}$ to the whole space $X$ defined by
$\widetilde{{u}_{0}}=\sum_{j=1}^{+\infty}e^{*}(n_{j})x_{n_{j}}^{*}$
and the elements $\widetilde{{w}_{0}}=x^{*}-\widetilde{{u}_{0}}$ and
$\widetilde{w}=\frac{\widetilde{{w}_{0}}}{\left\Vert
w_{0}\right\Vert }$. Now, by adapting to our framework the approach
developed in the last part of the proof of Theorem 8 in
\cite{Japon-Prus2004}, we show that the $w^{*}$-compact, convex set
\[
C=\left\{
\mu_{1}x^{*}+\mu_{2}\widetilde{w}+\sum_{j=1}^{+\infty}\mu_{j+2}x_{n_{j}}^{*}:\,\sum_{k=1}^{+\infty}\mu_{k}=1,\,\mu_{k}\geq0,\,
k=1,2,\dots\right\}
\]
can be rewritten as
\[
C=\left\{
\lambda_{1}\widetilde{w}+\sum_{j=1}^{+\infty}\lambda_{j+1}x_{n_{j}}^{*}:\,\sum_{k=1}^{+\infty}\lambda_{k}=1,\,\lambda_{k}\geq0,\,
k=1,2,\dots\right\} .
\]
Now, we consider the map $T:C\rightarrow C$ defined by:
\[
T\left(\lambda_{1}\widetilde{w}+\sum_{j=1}^{+\infty}\lambda_{j+1}x_{n_{j}}^{*}\right)=\sum_{j=1}^{+\infty}\lambda_{j}x_{n_{j}}^{*}.
\]
Since
$x=\lambda_{1}\widetilde{w}+\sum_{j=1}^{+\infty}\lambda_{j+1}x_{n_{j}}^{*}\in
C$ has a unique representation the map $T$ is well defined. Moreover
it is a nonexpansive map. Indeed, for every
$\alpha_{j}\in\mathbb{R}$, $j=1,2,\dots$ it holds
\[
\left\Vert
\alpha_{1}\widetilde{w}+\sum_{j=1}^{+\infty}\alpha_{j+1}x_{n_{j}}^{*}\right\Vert
\geq\left\Vert \alpha_{1}\frac{w_{0}}{\left\Vert w_{0}\right\Vert
}+\sum_{j=1}^{+\infty}\alpha_{j+1}e_{n_{j}}^{*}\right\Vert
=\sum_{j=1}^{+\infty}\left|\alpha_{j}\right|
\]
\[
=\sum_{j=1}^{+\infty}\left|\alpha_{j}\right|\left\Vert
x_{n_{j}}^{*}\right\Vert \geq\left\Vert
\sum_{j=1}^{+\infty}\alpha_{j}x_{n_{j}}^{*}\right\Vert .
\]
Finally, it is easy to see that $T$ has not a fixed point in
$C$.
\end{proof}

As already pointed out in Remark \ref{remark quotient c} for Theorem \ref{thm c inside predual}, we can extend also Theorem \ref{W_f in a predual of X} by assuming a property of the quotients of $X$.

\begin{rem}\label{W_f + quotient}
Let $X$ be a separable Banach space and let us suppose that a "bad"
$W_f$ is a subspace of a quotient $X/Y$ of $X$. Theorem \ref{W_f in
a predual of X} shows that $Y^{\bot}$ fails the $w^{*}$-FPP. It is
straightforward to see that also $X^{*}$ fails the $w^{*}$-FPP.
\end{rem}

In the next section we will see that the property stated in previous remark becomes a necessary condition if we additionaly assume that X is a separable Lindenstrauss space. Indeed, we will prove that  the lack of $w^*$-FPP implies the existence of a quotient of $X$  containing a "bad" $W_f$.

\section{The case of separable Lindenstrauss spaces}\label{Section OGT}
This section is devoted to the main result of our paper. We provide
a characterization of separable Lindenstrauss spaces $X$ that induce
a $w^*$-topology such that $\sigma(X^*,X)$-FPP fails.

By taking in account Corollary \ref{cor-nonseparable Lspace} we can  limit ourselves to study the Lindenstrauss spaces
whose dual is isometric to $\ell_1$.

It is worth pointing out that the sufficient condition
for the failure of $\sigma (X^*,X)$-FPP stated in Remark \ref{W_f +
quotient}, reveals to be also necessary. This fact emphasizes the
crucial role played in the study of $w^*$-FPP by the spaces "bad"
$W_f$. Moreover, we are able to find also a condition involving the
limit of a $w^*$-convergent subsequence of the standard basis of
$\ell_1$ that is equivalent to the failure of $\sigma
(\ell_1,X)$-FPP. This property allows us to give a characterization of
the $w^*$-FPP in $\ell_1$ by removing the restrictive assumption
about the convergence of the standard basis of $\ell_1$ used in
Theorem 8 in \cite{Japon-Prus2004}.
\begin{thm}\label{OGT}
Let $X$ be a predual of $\ell_1$. Then the following are equivalent.

\begin{enumerate}
\item \label{item noFPP} $\ell_1$ lacks the $\sigma(\ell_1,X)$-FPP for nonexpansive mappings.

\item \label{item isometry}$\ell_1$ lacks the $\sigma(\ell_1,X)$-FPP for isometries.

\item \label{item contractive}$\ell_1$ lacks the $\sigma(\ell_1,X)$-FPP for contractive mappings.

\item \label{item subsequence}There is a subsequence $(e_{n_k}^*)_{k \in
\mathbb{N}}$ of the standard basis $(e_{n}^*)_{n \in \mathbb{N}}$ in
$\ell_1$ which is $\sigma(\ell_1,X)$-convergent to a norm-one
element $e^* \in \ell _1$ with $e^*(n_k)\geq 0$ for all $k\in
\mathbb{N}$.

\item\label{item quotient isometric} There is a quotient of $X$ isometric to a ``bad'' $W_f$.

\item \label{item quotient contained}There is a quotient of $X$ that contains a subspace isometric to a ``bad'' $W_g$.

\end{enumerate}
\end{thm}

\begin{proof}
We divide the proof of this theorem in several parts. First of all
we remark that some implications are straightforward to prove.
Indeed, it is easy to check that $(\ref{item
isometry})\Longrightarrow (\ref{item noFPP})$, $(\ref{item
contractive})\Longrightarrow (\ref{item noFPP})$ and $(\ref{item
quotient isometric})\Longrightarrow (\ref{item quotient
contained})$. The implication $(\ref{item quotient
contained})\Longrightarrow (\ref{item noFPP})$ follows immediately
from Remark \ref{W_f + quotient}.

$(\ref{item subsequence})\Longrightarrow (\ref{item isometry})$ and
$(\ref{item subsequence})\Longrightarrow (\ref{item contractive})$.
By adapting to our setting the method developed in the last part of
Theorem 8 in \cite{Japon-Prus2004}, we obtain a $w^*$-compact and
convex set $C\subset \ell_1$ and an isometry $T:C\rightarrow C$
fixed point free.  Moreover, by following the idea of
\cite{Burns-Lennard-Sivek}, we consider the mapping $S:C\rightarrow
C$ defined as
\[
S(x)=\sum _{j=0}^{\infty}\frac{T^j(x)}{2^{j+1}},
\]
where $T$ is as above. It is easy to prove that the mapping $S$ is a
fixed point free contractive mapping.

$(\ref{item subsequence}) \Longrightarrow (\ref{item quotient isometric}).$
By choosing a subsequence we may assume that
$u^*=e^*-\sum_{k=2}^{\infty}e^*(n_k) e^*_{n_k}\neq 0$. Put
$x_1^*=\frac{u^*}{\left\|u^*\right\|}$ and $x_k^*=e^*_{n_k}$ for $k
\geq 2$. It is easy to see that $(x_k^*)_{k \in \mathbb{N}}$ is
normalized sequence which is equivalent to the standard basis in
$\ell_1$. Let us denote by $Y=\left[\left\{x_k^*: k\in
\mathbb{N}\right\}\right]$ the closed linear span of $\left\lbrace x_k^*: k\in
\mathbb{N}\right\rbrace $. Since $\overline{\left\{x_k^*: k\in
\mathbb{N}\right\}}^{\textrm{w}^*}=\left\{x_k^*: k\in
\mathbb{N}\right\}\cup \left\{e^*\right\}\subset Y$, Lemma 1 in
\cite{A1992} guarantees that $\overline{\left[\left\{x_k^*: k\in
\mathbb{N}\right\}\right]}^{\textrm{w}^*}=Y$. Let us consider $W_f \subset
c$ where
\[
f=\left(-\frac12, \frac12
\left(1-\sum_{k=2}^{\infty}e^*(n_k)\right),\frac12 e^*(n_2), \frac12
e^*(n_3), \frac12 e^*(n_4),\dots \right).
\]
Then, by recalling Definition \ref{definition bad W_f}, we have that $W_f$ is
a "bad" $W_f$. Let $(y_n^*)_{n \in \mathbb{N}}$ denote the
standard basis in $\ell_1=W_f^*$. We shall consider two cases.
Suppose $\sum_{k=2}^{\infty}e^*(n_k)>0$. Then, applying Theorem 4.3
in \cite{Casini-Miglierina-Piasecki2014}, we obtain $
y_{n}^*\overset{\sigma(\ell_1,W_f)}{\longrightarrow}y^*$, where
$y^*=\left(1-\sum_{k=2}^{\infty}e^*(n_k),e^*(n_2),e^*(n_3),e^*(n_4),\dots\right)$.
Let $\phi$ be the basis to basis map of $Y$ onto $\ell_1=W_f^*$,
$\phi\left(\sum_{k=1}^{\infty}a_k
x_k^*\right)=\sum_{k=1}^{\infty}a_k y_k^*.$ Then we have
\begin{eqnarray*}
\phi(e^*)&=&\phi \left(u^*+ \sum_{k=2}^{\infty} e^*(n_k)
e^*_{n_k}\right)=\phi \left(\left\|u^*\right\| x_1^*+
\sum_{k=2}^{\infty} e^*(n_k) x^*_k\right) \\&=&\left\|u^*\right\|
y_1^*+ \sum_{k=2}^{\infty} e^*(n_k) y^*_k =
\left(1-\sum_{k=2}^{\infty}e^*(n_k),e^*(n_2),e^*(n_3),\dots\right)=y^*.
\end{eqnarray*}
Consequently, $\phi$ is a $w^*$-continuous homeomorphism from
$\overline{\left\{x_k^*: k\in \mathbb{N}\right\}}^{\textrm{w}^*}$
onto $\overline{\left\{y_k^*: k\in
\mathbb{N}\right\}}^{\textrm{w}^*}=\left\{y_k^*: k\in
\mathbb{N}\right\} \cup \left\{y^*\right\}$. So, in view of Lemma 2
in \cite{A1992} we see that $\phi$ is a $w^*$-continuous isometry
from $Y$ onto $\ell_1=W_f^*$. This implies that $W_f$ is isometric
to $X/^{\bot}Y$.
Finally, if $\sum_{k=2}^{\infty}e^*(n_k)=0$ then $W_f$ is isometric to $c$. By following the same reasoning as above, we easily conclude that $c$ is isometric to a quotient of $X$.

To conclude the proof it remains to show that (\ref{item noFPP})$\Longrightarrow$(\ref{item subsequence}). This part is the key point of the whole proof and we split it in several steps for the sake of convenience of the reader.

$(\ref{item noFPP})\Longrightarrow(\ref{item subsequence})$.\textbf{The Final Step.} Suppose that we have already constructed a
sequence $(x_m)_{m \in \mathbb{N}} \subset B_X$, a
$\sigma(\ell_1,X)$-convergent subsequence $(e_{n_k}^*)_{k\in
\mathbb{N}}$ of the standard basis $(e_{n}^*)_{n\in \mathbb{N}}$ in
$\ell_1=X^*$ and a null sequence $(\varepsilon_m)_{m \in
\mathbb{N}}$ in $(0,1)$ such that for all $k,m \in \mathbb{N}$ we
have $e_{n_k}^*(x_m)>1-\varepsilon_m.$ If $e^*$ denotes the
$\sigma(\ell_1,X)$-limit of $(e_{n_k}^*)_{k\in \mathbb{N}}$, then
$\left\|e^*\right\|=1$ and $e^*(n_k)\geq 0$ for all $k\in
\mathbb{N}$.

Indeed, let $k_0 \in \mathbb{N}$ be arbitrarily chosen. Since
$e^*_{n_k}(x_m)\overset{k}{\longrightarrow}e^*(x_m)$, we get
$e^*(x_m)\geq 1-\varepsilon_m$. Consequently, for each $m \in
\mathbb{N}$, we have
$$e^*_{n_{k_0}}(x_m)+e^*(x_m)>1-\varepsilon_m+1-\varepsilon_m=2-2\varepsilon_m.$$
Hence, $\left\|e^*_{n_{k_0}}+e^*\right\|\geq 2$, from which our
thesis follows at once.

\textit{In the sequel we present how to construct sequences
$(x_m)_{m \in \mathbb{N}}$, $(e_{n_k}^*)_{k\in \mathbb{N}}$ and
$(\varepsilon_m)_{m \in \mathbb{N}}$ described above.}

\textbf{Step 1.} \textit{The sequence $(x_n^*)_{n \in \mathbb{N}\cup
\left\{0\right\}}$.} Assume that $\ell_1$ lacks the
$\sigma(\ell_1,X)$-FPP. Then, from the proof of Theorem 8 in
\cite{Japon-Prus2004}, we know that there is a sequence $(x_n^*)_{n
\in \mathbb{N}\cup \left\{0\right\}}$ in $\ell_1$ with the following
properties:

\begin{enumerate}
  \item[(i)] $x_n^*\overset{\sigma(\ell_1,X)}{\longrightarrow}x_0^*$,
  \item[(ii)] $\left(x_n^*\right)_{n \in \mathbb{N}}$ tends to $0$ coordinatewise,
  \item[(iii)] $\lim\limits_{n\rightarrow
  \infty}\left\|u^*-x_n^*\right\|=2$ for every $u^* \in \textrm{conv} \left\{x_n^*:
  n\geq0\right\}$,
  \item[(iv)] $\lim\limits_{n\rightarrow
  \infty}\left\|x_n^*\right\|=1=\left\|x_0^*\right\|$.
\end{enumerate}

Now, using (ii), (iii) and (iv), one may observe that for every
$n\in \mathbb{N}$ we have
$$2=\lim\limits_{m\rightarrow \infty}\left\|x_n^*-x_m^*\right\|=\left\|x_n^*\right\|+
\lim\limits_{m\rightarrow
\infty}\left\|x_m^*\right\|=\left\|x_n^*\right\|+1$$ and,
consequently,
\begin{enumerate}
  \item[(v)] $\left\|x_n^*\right\|=1$ for all $n\geq 0$.
\end{enumerate}
Again, using (ii), (iii) and (iv), one may notice that for all $m,n
\in \mathbb{N}\cup \left\{0\right\}$ we obtain
$$2=\lim\limits_{k\rightarrow \infty}\left\|\frac12\left(x_n^*+x_m^*\right)-x_k^*\right\|=
\left\|\frac12\left(x_n^*+x_m^*\right)\right\|+
\lim\limits_{k\rightarrow
\infty}\left\|x_k^*\right\|=\left\|\frac12\left(x_n^*+x_m^*\right)\right\|+1,$$
hence,
\begin{enumerate}
  \item[(vi)] $\left\|x_n^*+x_m^*\right\|=2$ for all $m,n
\in \mathbb{N}\cup \left\{0\right\}$.
\end{enumerate}
Taking into account (v) and (vi) we easily conclude that
\begin{enumerate}
  \item[(vii)] $x_n^*(i)\cdot x_m^*(i)\geq0$ for all $m,n
\in \mathbb{N}\cup \left\{0\right\}$ and for every $i \in
\mathbb{N}$.
\end{enumerate}

From now on we set $\sum_{i\in \emptyset}a_i:=0$.

\textbf{Step 2.} \textit{Grinding of the sequence $(x_n^*)_{n \in
\mathbb{N}\cup \left\{0\right\}}$.} Let $(x_n^*)_{n \in
\mathbb{N}\cup \left\{0\right\}}$ be as above. We show that there is
a sequence $(y_k^*)_{k \in \mathbb{N}\cup \left\{0\right\}}$ in
$\ell_1$ and numbers $s^+ \in (0,1]$, $s^- \in (-1,0]$ such that

\begin{enumerate}
  \item[(a)] $\left\|y_k^*\right\|=1$ for every $k\geq 0$,
  \item[(b)]  for every $k\in \mathbb{N}$ the set $\textrm{supp } y_k^*:=\left\{i\in \mathbb{N}:y_k^*(i)\neq 0
  \right\}$ is finite and $\max \textrm{supp } y_k^* < \min \textrm{supp } y_{k+1}^*$,
  \item[(c)] $y_m^*(i)\cdot y_n^*(i)\geq0$ for all $m,n
\in \mathbb{N}\cup \left\{0\right\}$ and for every $i \in
\mathbb{N}$,
  \item[(d)] for every $k\in \mathbb{N}$
  $$s^+(y_k^*):= \sum_{i\in \textrm{supp$_+$ } y_{k}^*}y_k^*(i)=s^+
  \quad \textrm{and} \quad s^-(y_k^*):= \sum_{i\in \textrm{supp$_-$ } y_{k}^*}y_k^*(i)=s^-,$$
  where $\textrm{supp$_+$ } y_{k}^*:= \left\{i\in \mathbb{N}:y_k^*(i)>0
  \right\}$, $\textrm{supp$_-$ } y_{k}^*:= \left\{i\in \mathbb{N}:y_k^*(i)<0
  \right\},$
  \item[(e)] $y_k^*\overset{\sigma(\ell_1,X)}{\longrightarrow}y_0^*$.
\end{enumerate}

Indeed, using (ii) and (v), we can choose a subsequence
$(x_{n_k}^*)_{k \in \mathbb{N}}$ of $(x_n^*)_{n \in \mathbb{N}}$ and
a sequence $(m_k)_{k \in \mathbb{N}\cup \left\{0\right\}}$, $m_k \in
\mathbb{N}\cup \left\{0\right\}$,
$0=m_0<m_1<m_2<\dots<m_k<m_{k+1}<\dots$, such that for every $k\in
\mathbb{N}$
\begin{equation}\label{inequlity for x_n_k}
\sum_{i=m_{k-1}+1}^{m_k}\left|x_{n_k}^*(i)\right|>1-\frac{1}{2^k}.
\end{equation}
Now, for every $k\in \mathbb{N}$ we put
$\widetilde{x_{n_k}^*}=\sum_{i=m_{k-1}+1}^{m_k}x_{n_k}^*(i)e_i^*$
and
$\widetilde{\widetilde{x_{n_k}^*}}=\frac{\widetilde{x_{n_k}^*}}{\left\|\widetilde{x_{n_k}^*}\right\|}.$
We can assume that the limits $\lim\limits_k
s^+\left(\widetilde{\widetilde{x_{n_k}^*}}\right)$ and
$\lim\limits_k s^-\left(\widetilde{\widetilde{x_{n_k}^*}}\right)$
exist, and let $s_0^+:=\lim\limits_k
s^+\left(\widetilde{\widetilde{x_{n_k}^*}}\right)$ and
$s_0^-:=\lim\limits_k
s^-\left(\widetilde{\widetilde{x_{n_k}^*}}\right).$ Clearly,
$s_0^+\in [0,1]$, $s_0^-\in [-1,0]$ and $s_0^+-s_0^-=1$. We shall
consider two cases.

First, suppose $s_0^+>0$. Then we can assume that
$s^+\left(\widetilde{\widetilde{x_{n_k}^*}}\right)>0$ for all $k\in
\mathbb{N}$. Further, suppose $s_0^-<0$. Then we can also assume
that $s^-\left(\widetilde{\widetilde{x_{n_k}^*}}\right)<0$ for all
$k\in \mathbb{N}$. Define the sequence $(y_k^*)_{k \in
\mathbb{N}\cup \left\{0\right\}}$ as $y_0^*=x_0^*$ and for $k\in
\mathbb{N}$

\begin{equation*}
y_k^*:=\frac{s_0^+}{s^+\left(\widetilde{\widetilde{x_{n_k}^*}}\right)}
\cdot \sum_{i \in \textrm{supp$_+$
}\widetilde{\widetilde{x_{n_k}^*}}}\widetilde{\widetilde{x_{n_k}^*}}(i)\cdot
e_i^*+\frac{s_0^-}{s^-\left(\widetilde{\widetilde{x_{n_k}^*}}\right)}
\cdot \sum_{i \in \textrm{supp$_-$
}\widetilde{\widetilde{x_{n_k}^*}}}\widetilde{\widetilde{x_{n_k}^*}}(i)\cdot
e_i^*.
\end{equation*}
Obviously, conditions (a), (b) and (c) are satisfied. Moreover,
$s^+(y_k^*)=s_0^+$ and $s^-(y_k^*)=s_0^-$, so in order to obtain (d)
it is enough to take $s^+=s_0^+$ and $s^-=s_0^-$. We shall prove
that (e) holds too. Indeed, by considering (\ref{inequlity for
x_n_k}), (i) and (v), we get $\lim\limits_k
\left\|\widetilde{x_{n_k}^*}\right\|=1$, $w^*\textrm{-}\lim\limits_k
\textrm{ } (\sum_{i=m_{k-1}+1}^{m_k}x_{n_k}^*(i) \cdot e_i^*)=x_0^*$
and, consequently, $w^*\textrm{-}\lim\limits_k y_k^*=x_0^*=y_0^*$,
as we desired.

If $s_0^-=0$, then $s_0^+=1$ and we can assume that
$s^+\left(\widetilde{\widetilde{x_{n_k}^*}}\right)>0$ for all $k\in
\mathbb{N}$. We define the sequence $(y_k^*)_{k \in \mathbb{N}\cup
\left\{0\right\}}$ as $y_0^*=x_0^*$ and for $k\in \mathbb{N}$ we put
$$y_k^*=\frac{s_0^+}{s^+\left(\widetilde{\widetilde{x_{n_k}^*}}\right)}
\cdot \sum_{i \in \textrm{supp$_+$
}\widetilde{\widetilde{x_{n_k}^*}}}\widetilde{\widetilde{x_{n_k}^*}}(i)\cdot
e_i^*.$$ It is easy to see that the properties (a), (b), (c), (e),
and (d) with $s^+:=s_0^+=1$ and $s^-:=s_0^-=0$ are satisfied.

Suppose $s_0^+=0$. Then $s_0^-=-1$ and we can assume that
$s^-\left(\widetilde{\widetilde{x_{n_k}^*}}\right)<0$ for all $k\in
\mathbb{N}$. Now, it is enough to define $(y_k^*)_{k \in
\mathbb{N}\cup \left\{0\right\}}$ as $y_0^*=-x_0^*$ and for $k\in
\mathbb{N}$
$$y_k^*=-\frac{s_0^-}{s^-\left(\widetilde{\widetilde{x_{n_k}^*}}\right)}
\cdot \sum_{i \in \textrm{supp$_-$
}\widetilde{\widetilde{x_{n_k}^*}}}\widetilde{\widetilde{x_{n_k}^*}}(i)\cdot
e_i^*.$$ Then, for every $k\in \mathbb{N}$, $s^+(y_k^*)=-s_0^-=1$.
Obviously, all the properties (a), (b), (c), (e) and (d) with
$s^+=1$ and $s^-=0$ are satisfied.

\textbf{Step 3.} \textit{A construction of the sequence $(x_m)_{m
\in \mathbb{N}}$.} Let $(y_k^*)_{k \in \mathbb{N}\cup
\left\{0\right\}}$, $s^-\in(-1,0]$ and $s^+ \in (0,1]$ be as above.
By using (a) and (e), we can choose $x_1 \in B_X$ and $k_1\in
\mathbb{N}$ such that $y_0^*(x_1)>1-\frac{s^+}{8}$ and
$y_k^*(x_1)>1-\frac{s^+}{8}$ for all $k\geq k_1$. Next, using (a)
and (c) we can choose $x_2 \in B_X$ such that
$y_0^*(x_2)>1-\frac{s^+}{8^2}$ and
$y_{k_1}^*(x_2)>1-\frac{s^+}{8^2}$. Moreover, the property (e)
implies that there is $k_2>k_1$ such that for all $k\geq k_2$ we
have $y_k^*(x_2)>1-\frac{s^+}{8^2}$. Further, using (a), (c) and
(e), we can choose $x_3\in B_X$ and $k_3>k_2$ such that
$y_0^*(x_3)>1-\frac{s^+}{8^3}$, $y_{k_1}^*(x_3)>1-\frac{s^+}{8^3}$,
$y_{k_2}^*(x_3)>1-\frac{s^+}{8^3}$ and
$y_k^*(x_3)>1-\frac{s^+}{8^3}$ for all $k\geq k_3$. Continuing this
inductive procedure, we construct a sequence $(x_m)_{m\in
\mathbb{N}}\subset B_X$ and a subsequence $(y_{k_n}^*)_{n \in
\mathbb{N}}$ of $(y_k^*)_{k \in \mathbb{N}\cup \left\{0\right\}}$
such that $y_{k_n}^*(x_m)>1-\frac{s^+}{8^m}$ for all $m,n\in
\mathbb{N}$.

For each $n\in \mathbb{N}$ we put $z_n^*=y_{k_n}^*$. Then the
sequence $(z_n^*)_{n \in \mathbb{N}}$ has the following properties:

\begin{enumerate}
  \item[(a')] for every $n\in \mathbb{N}$ the set $\textrm{supp$_+$ } z_n^*$ is nonempty, $\textrm{supp } z_n^*$ is finite, and
  $\max \textrm{supp } z_n^* < \min \textrm{supp } z_{n+1}^*$,
  \item[(b')] $z_n^*(x_m)>1-\frac{s^+}{8^m}$ for all $m,n \in
\mathbb{N}$,
  \item[(c')] for every $n\in \mathbb{N}$, $s^+(z_n^*)=s^+$ and $s^-(z_n^*)=s^-$.
\end{enumerate}

\textbf{Step 4.} \textit{A construction of the sequences
$(e_{n_k}^*)_{k\in \mathbb{N}}$ and $(\varepsilon_m)_{m \in
\mathbb{N}}$.} Let $(z_n^*)_{n \in \mathbb{N}}$, $s^-$ and $s^+$ be
as above. For each $m,n \in \mathbb{N}$ we define the set
$$E_m^{(n)}=\left\{ i\in \textrm{supp$_+$ } z_{n}^*: e_i^* (x_m)>1- \frac{1}{2^m}\right\}.$$
Then, using (c'), we have

\begin{eqnarray*}
\sum_{i\in \textrm{supp$_+$ } z_{n}^*}z_{n}^*(i)\cdot e_i^* (x_m)
&=&\sum_{i\in E_m^{(n)}}z_{n}^*(i)\cdot e_i^* (x_m)+\sum_{i\in
(\textrm{supp$_+$ } z_{n}^*) \setminus E_m^{(n)}}z_{n}^*(i)\cdot
e_i^* (x_m)
\\&\leq&\sum_{i\in E_m^{(n)}}z_{n}^*(i)+
\left(1-\frac{1}{2^m}\right) \cdot \sum_{i\in (\textrm{supp$_+$ }
z_{n}^*) \setminus E_m^{(n)}}z_{n}^*(i)
\\&=&\left(1-\frac{1}{2^m}\right) \cdot \sum_{i\in \textrm{supp$_+$ }
z_{n}^*}z_{n}^*(i)+ \frac{1}{2^m} \cdot \sum_{i\in
E_m^{(n)}}z_{n}^*(i)
\\&=& \left(1-\frac{1}{2^m}\right) \cdot s^+ +\frac{1}{2^m} \cdot \sum_{i\in
E_m^{(n)}}z_{n}^*(i).
\end{eqnarray*}
On the other hand, using (b') and (c'), we get
\begin{eqnarray*}
\sum_{i\in \textrm{supp$_+$ } z_{n}^*}z_{n}^*(i)\cdot e_i^* (x_m)
&=&\sum_{i\in \textrm{supp } z_{n}^*}z_{n}^*(i)\cdot e_i^* (x_m)-
\sum_{i\in \textrm{supp$_-$ } z_{n}^*}z_{n}^*(i)\cdot e_i^* (x_m)
\\&>&
1-\frac{s^+}{8^m}+s^-=1-\frac{s^+}{8^m}-1+s^+=\left(1-\frac{1}{8^m}\right)\cdot
s^+.
\end{eqnarray*}
The above implies that
$$\left(1-\frac{1}{8^m}\right)\cdot s^+< \left(1-\frac{1}{2^m}\right) \cdot s^+ +\frac{1}{2^m} \cdot \sum_{i\in
E_m^{(n)}}z_{n}^*(i),$$ so
$$\sum_{i\in
E_m^{(n)}}z_{n}^*(i) \geq \left( 1- \frac{1}{4^m}\right) \cdot
s^+.$$ The above calculations also show that for any $m,n \in
\mathbb{N}$ the set $E_m^{(n)}$ is nonempty and
\begin{equation}\label{evaluation 1}
\sum_{i\in (\textrm{supp$_+$ } z_{n}^*) \setminus
E_m^{(n)}}z_{n}^*(i) \leq \frac{1}{4^m}\cdot s^+.
\end{equation}

For each $m,n\in \mathbb{N}$ we define the set
$F_m^{(n)}=\bigcap\limits_{j=1}^m E_j^{(n)}$. Obviously, for every
$n\in \mathbb{N}$, $F_1^{(n)}\supseteq F_2^{(n)}\supseteq
F_3^{(n)}\supseteq \dots$. We claim that for every $n \in
\mathbb{N}$ and $m \in \mathbb{N}$ the set $F_m^{(n)}$ is nonempty.
Indeed, suppose that there is $n \in \mathbb{N}$ and $m \in
\mathbb{N}$ such that $F_m^{(n)}=\emptyset$. Then, $\textrm{supp$_+$
} z_{n}^*=\bigcup\limits_{j=1}^m (\textrm{supp$_+$ } z_{n}^*
\setminus E_j^{(n)})$, so taking into account (\ref{evaluation 1})
and (c'), we obtain a contradiction,
$$\frac12 s^+> \sum_{j=1}^{m}\frac{1}{4^j}s^+ \geq \sum_{j=1}^{m}
\sum_{i\in (\textrm{supp$_+$ } z_{n}^*) \setminus
E_j^{(n)}}z_{n}^*(i)\geq \sum_{i\in \textrm{supp$_+$ }
z_{n}^*}z_{n}^*(i)=s^+.$$

Since for every $m\in \mathbb{N}$ and $n\in \mathbb{N}$ the set
$F_m^{(n)}$ is nonempty and, in view of (a'), $F_1^{(n)}$ is finite,
we conclude that the set $G^{(n)}:=\bigcap\limits_{m=1}^{\infty}
F_m^{(n)}$ is nonempty, for every $n\in \mathbb{N}$. Clearly,
$$G^{(n)}=\bigcap\limits_{m=1}^{\infty}
E_m^{(n)}=\left\{i\in \textrm{supp$_+$ } z_{n}^*:
e_i^*(x_m)>1-\frac{1}{2^m} \textrm{ for all } m\in
\mathbb{N}\right\}.$$ Moreover, using (a') we see that $G^{(i)}\cap
G^{(j)}=\emptyset$ provided $i\neq j$. Hence, the set
$\bigcup\limits_{j=1}^{\infty}G^{(j)}$ is infinite. Take any
$\sigma(\ell_1,X)$-convergent subsequence $(e_{n_k}^*)_{k\in
\mathbb{N}}$ of $(e_n^*)_{n\in
\bigcup\limits_{j=1}^{\infty}G^{(j)}}$. Then, for every $m\in
\mathbb{N}$ and $k\in \mathbb{N}$, we have
$e_{n_k}^*(x_m)>1-\frac{1}{2^m}$. Apply now \textbf{The Final Step},
with $\varepsilon_m=\frac{1}{2^m}$. The proof of $(\ref{item
noFPP})\Longrightarrow (\ref{item subsequence})$ is finished.

\end{proof}

\begin{rem}
The spaces "bad" $W_f$ and $W_g$ in the statements $(\ref{item quotient isometric})$ and $(\ref{item quotient contained})$ of Theorem \ref{OGT} cannot be replaced by the space $c$ (see Example \ref{example bad W-f}).
\end{rem}


We conclude the paper by pointing out an issue related to our results that still remain as open problem.
Let $X$ be a predual of
$\ell_1$. Theorem \ref{W_f in a predual of X} implies that the existence of an isometric copy of a "bad" $W_f$ in $X$ ensures the failure of $\sigma(\ell_1,X)$-FPP. On the other hand, Theorem \ref{OGT} provides a necessary and sufficient condition for the failure of $\sigma(\ell_1,X)$-FPP based on the existence of a quotient of $X$ isometric to a "bad" $W_f$. Taking into account these two facts, one natural question still unanswered is whether the lack of $\sigma(\ell_1,X)$-FPP implies that $X$ contains an
isometric copy of a "bad" $W_f$.


\begin{thebibliography}{1}


\bibitem{A1992} D. E. Alspach, {\it A $\ell_{1}$-predual which is not
isometric to a quotient of $C(\alpha)$}, arXiv:math/9204215v1
{[}math.FA{]} 27 Apr 1992.

\bibitem{A1981} D. E. Alspach, {\it A fixed point free nonexpansive map},
Proc. Amer. Math. Soc. 82 (1981), 423-424.

\bibitem{Burns-Lennard-Sivek} J. Burns, C. Lennard and J. Sivek, {\it A contractive fixed point free mapping on a
weakly compact convex set}, Studia Math. 223 (2014), 275-283.


\bibitem{Casini-Miglierina-Piasecki2014} E. Casini, E. Miglierina and
\L. Piasecki, {\it Hyperplanes in the space of convergent sequences and
preduals of $\ell_1$}, online first on Canad. Math. Bull., 
http://dx.doi.org/10.4153/CMB-2015-024-9, 4 Mar 2015.

%


\bibitem{Goebel-Kuczumow1978} K. Goebel and T. Kuczumow, {\it Irregular
convex sets with the fixed point property for nonexpansive mappings},
Colloq. Math. 40 (1978), 259-264.


\bibitem{Japon-Prus2004}M. A. Jap\'on-Pineda and S. Prus, {\it Fixed point
property for general topologies in some Banach spaces}, Bull.
Austral. Math. Soc. 70 (2004), 229-244.

\bibitem{Karlovitz1976} L. A. Karlovitz, {\it On nonexpansive mappings}, Proc. Amer. Math. Soc. 55
(1976), 321-325.


%


\bibitem{Lazar-Lindenstrauss1971} A. J. Lazar and J. Lindenstrauss,
{\it Banach spaces whose duals are $L_1$-spaces and their representing
matrices}, Acta Math. 126 (1971), 165-194.

\bibitem{Handbook} B. Sims, {\it Examples of fixed point free mappings}, in: Handbook of Metric
Fixed Point Theory. Kluwer Acad. Publ., Dordrecht, 2001, 35-48.

%
%

\bibitem{Smyth1994} M. Smyth, {\it Remarks on the weak star fixed point
property in the dual of $C(\Omega)$}, J. Math. Anal. Appl.
195 (1995), 294-306.



\end{thebibliography}
\end{document}